\RequirePackage{fix-cm}
\documentclass[11pt,english,twoside,openright]{article}
\usepackage{amsmath}
\usepackage{amsfonts,amssymb,amsthm}
\usepackage{bbm}
\usepackage{graphicx}
\usepackage{epstopdf}
\usepackage{enumerate}
\usepackage{enumitem}
\usepackage{geometry}
\usepackage{framed}
\usepackage{lipsum}
\usepackage{setspace}
\usepackage{url}
\usepackage{algorithmic}
\usepackage{algorithm}
\usepackage[titletoc,title]{appendix}
\numberwithin{equation}{section}

\newcommand{\Real}{\mathbb{R}}
\newcommand{\bone}{{\bf 1}}
\newcommand{\bzero}{{\bf 0}}
\newcommand{\ba}{{\bf a}}
\newcommand{\bb}{{\bf b}}
\newcommand{\bc}{{\bf c}}
\newcommand{\bd}{{\bf d}}
\newcommand{\be}{{\bf e}}

\newcommand{\bp}{{\bf p}}

\newcommand{\bt}{{\bf t}}
\newcommand{\bu}{{\bf u}}
\newcommand{\bv}{{\bf v}}
\newcommand{\bw}{{\bf w}}
\newcommand{\bx}{{\bf x}}
\newcommand{\by}{{\bf y}}
\newcommand{\bz}{{\bf z}}

\newcommand{\beeta}{{\boldsymbol \eta}}

\newcommand{\bmu}{{\boldsymbol \mu}}

\newcommand{\blambda}{{\boldsymbol \lambda}}

\newcommand{\bA}{{\bf A}}
\newcommand{\bB}{{\bf B}}

\newcommand{\bE}{{\bf E}}

\newcommand{\bI}{{\bf I}}

\newcommand{\bT}{{\bf T}}

\newcommand{\bV}{{\bf V}}
\newcommand{\bW}{{\bf W}}

\newcommand{\inprod}[2]{\left\langle{#1},{#2}\right\rangle}
\newcommand{\inprodmy}[2]{\langle{#1},{#2}\rangle}

\newcommand{\norm}[1]{\left\|{#1}\right\|}
\newcommand{\normmy}[1]{\|{#1}\|}

\newcommand{\lambdamin}[1]{\lambda_{\min}\left({#1}\right)}
\newcommand{\paren}[1]{\left\{{#1}\right\}}

\DeclareMathOperator{\conv}{conv}
\DeclareMathOperator{\ext}{ext}

\DeclareMathOperator{\rank}{rank}
\DeclareMathOperator{\image}{Im}

\DeclareMathOperator*{\argmin}{\arg\!\min}
\DeclareMathOperator*{\argmax}{\arg\!\max}
\newtheorem{theorem}{Theorem}[section]
\newtheorem{lemma}{Lemma}[section]

\newtheorem{remark}{Remark}[section]

\newtheorem{assumption}{\it Assumption} 
\newtheorem{defin}{Definition}[section]
\newtheorem{cor}{Corollary}[section]
\title{Linearly Convergent Away-Step Conditional Gradient for Non-strongly Convex Functions}
\author{Amir Beck\thanks{Faculty of Industrial Engineering and Management,
Technion - Israel Institute of Technology, Haifa, Israel. Email: becka@ie.technion.ac.il.} \and Shimrit Shtern\thanks{Faculty of Industrial Engineering and Management,
Technion - Israel Institute of Technology, Haifa, Israel. Email: shimrits@tx.technion.ac.il.}}
\begin{document}
\maketitle
\begin{abstract}
We consider the problem of minimizing a function, which is the sum of a linear function and a composition of a strongly convex function with a linear transformation, over a compact polyhedral set. Jaggi and Lacoste-Julien \cite{Jaggi2013} showed that the conditional gradient method with away steps employed on the aforementioned problem without the additional linear term has linear rate of convergence, depending on the so-called pyramidal width of the feasible set. We revisit this result and provide a variant of the algorithm and an analysis that is based on simple duality arguments, as well as corresponding error bounds. This new analysis (a) enables the incorporation of the additional linear term, (b) does not require a linear-oracle that outputs an extreme point of the linear mapping of the feasible set and (c) depends on a new constant,  termed ``the vertex-facet distance constant", which is explicitly expressed in terms of the problem's parameters and the geometry of the feasible set. This constant replaces the pyramidal width, which is difficult to evaluate.
\end{abstract}
\section{Introduction}
Consider the minimization problem
\begin{equation}\label{eq:Problem}\tag{P}
\begin{aligned}
\min_{\bx\in X} \left \{ f(\bx)\equiv g(\bE\bx)+ \inprod{\bb}{\bx} \right \},
\end{aligned}
\end{equation}
where $X \subseteq \Real^n$ is a compact polyhedral set, $\bE \in \Real^{m \times n}, \bb \in \Real^n$ and $g: \Real^m \rightarrow \Real$ is strongly convex and continuously differentiable over $\Real^m$.
Note that for a general matrix $\bE$, the function $f$ is not necessarily strongly convex.

When the problem at hand is large-scale, first order methods, which have relatively low computational cost per iteration, are usually utilized. These methods include, for example, the class of projected (proximal) gradient methods. A drawback of these methods is that under general convexity assumptions, they posses only a sublinear rate of convergence \cite{Nesterov2004,BT09}, while linear rate of convergence can be established only under additional conditions such as strong convexity of the objective function \cite{Nesterov2004}. Luo and Tseng \cite{LuoTseng1993} showed that the strong convexity assumption can be relaxed and replaced by an assumption on the existence of a local error bound, and under this assumption,  certain classes algorithms, which they referred to as ``feasible descent methods",  converge in an asymptotic linear time. The model (\ref{eq:Problem}) with assumptions on strong convexity of $g$, compactness and polyhedrality of $X$ was shown in \cite{LuoTseng1993} to satisfy the error bound. In~\cite{WangLin2014} Wang and Lin extended the work \cite{LuoTseng1993} and  showed that there exists a \textit{global} error bound for problem (\ref{eq:Problem}) with the additional assumption of compactness of $X$; and derived the exact linear rate for this case. We note that the family of ``feasible descent methods" include the block alternating minimization algorithm (under the assumption of block strong convexity), as well as gradient projection methods, and therefore are usually at least as complex as evaluating the orthogonal projection operator onto the feasible set $X$ at each iteration.

An alternative to algorithms which are based on projection (or proximal) operators are \emph{linear-oracle}-based algorithms such as the conditional gradient (CG) method. The CG algorithm was presented by Frank and Wolfe in 1956 \cite{FrankWolfe1956}, for minimizing a convex function over a compact polyhedral set. At each iteration, the algorithm requires a solution to the problem of minimizing a linear objective function over the feasible set. It is assumed that this solution is obtained by a call to a linear-oracle, i.e., a black box which, given a linear function, returns an optimal solution of this linear function over the feasible set (see an exact definition in Section~\ref{sec:CG}). In some instances, and specifically for certain types of polyhedral sets, obtaining such a linear-oracle can be done more efficiently than computing the orthogonal projection onto the feasible set (see examples in \cite{Garber2013}), and therefore the CG algorithm has an advantage over projection-based algorithms. The original paper of Frank and Wolfe also contained a proof of an $O(1/k)$ rate of convergence of the function values to the optimal value. Levitin and Polyak showed in \cite{levitin1966} that this $O(1/k)$ rate can also be extended to the case where the feasible set is a general compact convex set.  Cannon and Culum proved in \cite{CanonCulum1968} that this rate is in fact \textit{tight}. However, if in addition to strong convexity of the objective function, the optimal solution is in the interior of the feasible set, then linear rate of convergence of the CG method can be established\footnote{The paper \cite{Guelat1986} assumes that the feasible set is a  bounded polyhedral, but the proof  is actually correct for general compact convex sets.} \cite{Guelat1986}. Epelman and Freund \cite{EF00}, as well as Beck and Teboulle \cite{BT04} showed a linear rate of convergence of the conditional gradient with a special stepsize choice in the context of finding a point in the intersection of an affine space and a closed and convex set under a Slater-type assumption.  Another setting in which linear rate of convergence can be derived is when  the feasible set is uniformly (strongly) convex and the norm of the gradient of the objective function is bounded away from zero \cite{levitin1966}. \\
\indent Another approach for deriving a linear rate of convergence is to modify the algorithm. For example, Hazan and Garber used \textit{local} linear-oracles in \cite{Garber2013} in order to show linear rate of convergence of a ``localized" version of the conditional gradient method. A different modification, which is viable when the feasible set is a compact polyhedral, is to use a variation of the conditional gradient method that incorporates away steps. This version of the conditional gradient method, which we refer to as \emph{away steps conditional gradient} (ASCG), was initially suggested by Wolfe in \cite{Wolfe1970} and then studied by Guelat and Marcotte \cite{Guelat1986}, where a linear rate of convergence was established under the assumption that the objective function is strongly convex, as well as an assumption on the location of the optimal solution. In \cite{Jaggi2013} Jaggi and Lacoste-Julien were able to extend this result for the more general model  (\ref{eq:Problem}) for the case where $\bb=\bzero$, without restrictions on the location of the solution. We note that the ASCG requires that the linear-oracle will produce an optimal solution of the associated problem which is an extreme point. We will call such an oracle a \textit{vertex linear-oracle} (see the discussion in Section \ref{sec:VertexLinearOracles}).

\noindent {\bf Contribution.} In this work, our starting point and main motivation are the results of Jaggi and Lacoste-Julien \cite{Jaggi2013}. Our contribution is threefold:
\begin{itemize}
\item[(a)] We extend the results given in \cite{Jaggi2013} and show that the ASCG algorithm converges linearly for the general case of problem~\eqref{eq:Problem}, that is, for any value of $\bE$ and $\bb$. \\ \indent The additional linear term $\langle \bb,\bx \rangle$ enables us to consider much more general models. For example, consider the $l_1$-regularized least squares problem $\min_{\bx\in S} \{\norm{\bB \bx-\bc}^2+\lambda \|\bx\|_1\},$
where $S \subseteq \Real^n$ is a compact polyhedral, $\bB  \in \Real^{k \times n}, \bc \in \Real^k$ and $\lambda>0$. Since $S$ is compact, we can find a constant $M>0$ for which $\|\bx\|_1 \leq M$ for any $\bx \in S$. We can now rewrite the model as
$$\min_{\bx\in S, \|\bx\|_1 \leq y, y \in [0,M]} \norm{\bB \bx-\bc}^2+\lambda y,$$
which obviously fits the general model (\ref{eq:Problem})
\item[(b)] The analysis in \cite{Jaggi2013} assumes the existence of a  \textit{vertex} linear-oracle on the set $\bE X$, rather than an oracle for the set $X$. This fact is not significant for the ``pure" CG algorithm, since it only requires a linear-oracle and not a \textit{vertex} linear-oracle. This means that for the CG algorithm, a linear-oracle on $\bE X$ can be easily obtained by applying $\bE$ on the output of the linear-oracle on $X$. On the other hand, this argument fails for the ASCG algorithm that specifically requires the oracle to return an extreme point of the feasible set, and finding such a vertex linear-oracle on $\bE X$ might be a complex task , see Section \ref{sec:VertexLinearOracles} for more details. Our analysis only requires a vertex linear-oracle on the original set $X$.
    \item[(c)]  We present an analysis  based on simple duality arguments, which are completely different than the geometric arguments in \cite{Jaggi2013}. Consequently, we obtain a computable constant for the rate of convergence, which is explicitly expressed as a function of the problem's parameters and the geometry of the feasible set. This constant, which we call ``the vertex-facet distance constant",  replaces the so-called \emph{pyramidal width} constant from \cite{Jaggi2013}, which reflects the geometry of the feasible set and  is obtained as the optimal value of a very complex mixed integer saddle point optimization problem whose exact value is unknown even for simple polyhedral sets.
\end{itemize}

\noindent {\bf Paper layout.} The paper is organized as follows. Section~\ref{sec:Preliminaries} presents some preliminary results and definitions needed for the analysis. In particular, it provides a brief introduction to the classical CG algorithm and linear oracles. Section~\ref{sec:AwayStepConditionalGradient} presents the ASCG algorithm and the convergence analysis, and is divided into four subsections. In Section~\ref{sec:VertexLinearOracles} the concept of vertex linear-oracle, needed for the implementation of ASCG, is presented, and the difficulties of obtaining a vertex linear-oracle on a linear transformation of the feasible set are discussed. In Section~\ref{sec:ASCGmethod} we present the ASCG method with different possible stepsize choices. In Section~\ref{sec:RateConvergenceAnalysis}, we provide the rate of convergence analysis of the ASCG for problem~\eqref{eq:Problem}, and present the new \emph{vertex-facet distance} constant used in the analysis. Finally, in  Section~\ref{sec:FindingOmegaForPolyhedrons}, we demonstrate how to compute this new constant for a few examples of simple polyhedral sets.

{\bf Notations.}
We denote the cardinality of set $I$ by $|I|$. The difference, union and intersection of two given sets $I$ and $J$ are denoted by $I/J=\paren{a\in I: a\notin J}$, $I\cup J$ and $I\cap J$ respectively.
Subscript indices represent elements of a vector, while superscript indices represent iterates of the vector, i.e., $x_i$ is the $i$th element of vector $\bx$, $\bx^k$ is a vector at iteration $k$, and $x^k_i$ is the $i$th element of $\bx^k$. The vector $\be_i\in \Real^n$ is the $i$th vector of the standard basis of $\Real^n$, $\bzero\in \Real^n$ is the all-zeros vector, and $\bone\in \Real^n$ is the vector of all ones. Given two vectors $\bx,\,\by\in\Real^n$, their dot product is denoted by $\inprod{\bx}{\by}$. Given a matrix $\bA\in\Real^{m\times n}$ and vector $\bx\in\Real^n$, $\norm{\bA}$ denotes the spectral norm of $\bA$, and $\norm{\bx}$ denotes the $\ell_2$ norm of $\bx$, unless stated otherwise. $\bA^T$, $\rank(\bA)$ and $\image(\bA)$
represent the transpose, rank and image
of $\bA$ respectively. We denote the $i$th row of a given matrix $\bA$ by $\bA_i$, and given a set $I\subseteq \paren{1,\ldots,m}$, $\bA_I\in\Real^{|I|\times n}$ is the submatrix of $\bA$ such that $(\bA_I)_{j}=\bA_{I_j}$ for any $j=1,\ldots, |I|$. If $\bA$ is a symmetric matrix, then  $\lambdamin{\bA}$ is its minimal eigenvalue.
If a matrix $\bA$ is also invertible, we denote its inverse by $\bA^{-1}$. Given matrices $\bA\in\Real^{n\times m}$ and $\bB\in \Real^{n\times k}$, the matrix $[\bA,\bB]\in \Real^{n\times {(m+k)}}$ is their horizontal concatenation. Given a point $\bx$ and a closed convex set $X$, the distance between $\bx$ and $X$ is denoted by $d(\bx,X)=\min_{\by\in X}\norm{\bx-\by}$. The standard unit simplex in $\Real^n$ is denoted by $\Delta_n=\paren{\bx\in\Real_+^n: \inprod{\bone}{\bx}=1}$ and its relative interior by $\Delta^+_n=\paren{\bx\in\Real_{++}^n: \inprod{\bone}{\bx}=1}$. Given a set $X\subseteq\Real^n$, its convex hull is denoted by $\conv(X)$. Given a convex set $C$, the set of all its extreme points is denoted by $\ext(C)$.

\section{Preliminaries}\label{sec:Preliminaries}

\subsection{Mathematical Preliminaries}\label{sec:MathPreliminaries}
We start by presenting two technical lemmas. The first lemma is the well known \emph{descent lemma} which is fundamental in convergence rate analysis of first order methods. The second lemma is \emph{Hoffman's lemma} which is used in various error bound analyses over polyhedral sets.
\begin{lemma}[The Descent Lemma {\cite[Proposition A.24]{Bertsekas1999}}]\label{lemma:DescentLemma}
Let $f:\Real^n\rightarrow \Real$ be a continuously differentiable function with Lipschitz continuous gradient with constant $\rho$. Then for any $\bx,\by\in\Real^n$ we have
\begin{equation*}
f(\by)\leq f(\bx)+\inprod{\nabla f(\bx)}{\by-\bx}+\frac{\rho}{2}\norm{\bx-\by}^2
\end{equation*}
\end{lemma}

\begin{lemma}[Hoffman's Lemma \cite{Hoffman1952}]\label{lemma:Hoffman}
Let $X$ be a polyhedron defined by $X=\paren{\bx\in\Real^n:\bA\bx\leq \ba}$, for some $\bA\in\Real^{m\times n}$ and $\ba\in\Real^{m}$, and let $S=\paren{\bx\in\Real^n:\tilde{\bE}\bx=\tilde{\be}}$ where $\tilde{\bE}\in\Real^{r\times n}$ and $\tilde{\be}\in\Real^r$. Assume that $X\cap S\neq \emptyset$. Then, there exists a constant $\theta$, depending only on $\bA$ and $\tilde{\bE}$, such that any $\bx\in X$ satisfies
$$d{(\bx,X\cap S)}\leq \theta \norm{\tilde{\bE}\bx-\tilde{\be}}.$$
\end{lemma}
A complete and simple proof of this lemma is given in \cite[pg. 299-301]{Guler2010}. Defining $\mathcal{B}$ as the set of all matrices constructed by taking linearly independent rows from the matrix $\left[\tilde{\bE}^T,\bA^T\right]^T$, we can write $\theta$ as
$$\theta=\max_{\bB\in\mathcal{B}}\frac{1}{\lambdamin{\bB\bB^T}}.$$
We will refer to $\theta$ as the \emph{Hoffman constant} associated with matrix $\left[\tilde{\bE}^T,\bA^T\right]^T$.

\subsection{Problem's Properties}\label{sec:ProbProperties}
Throughout the article we make the following assumption regarding problem \eqref{eq:Problem}.
\begin{assumption}\label{ass:fx=gEx}
	\begin{enumerate}[label=(\alph*)] 
		\item\label{ass:f Lipschitz} $f$ is continuously differentiable and has a Lipschitz continuous gradient with constant $\rho$.
		\item\label{ass:g strongly convex} $g$ is strongly convex with parameter  $\sigma_g$.
		\item\label{ass:XPolyhedron} $X$ is a nonempty compact polyhedral set given by $X=\{\bx\in \Real^n: \bA\bx\leq \ba\}$ for some $\bA\in\Real^{m\times n}$, $\ba\in \Real^m$.
	\end{enumerate}
\end{assumption}

We denote the optimal solution set of problem \eqref{eq:Problem} by $X^*$.
The diameter of the compact set $X$ is denoted by $D$, and the diameter of the set $\bE X$ (the diameter of the image of $X$ under the linear mapping associated with matrix $\bE$) by $D_\bE$. The two diameters satisfy the following relation:
$$D_\bE=\max_{\bx,\by\in X}\norm{\bE\bx-\bE\by}\leq \norm{\bE}\max_{\bx,\by\in X }\norm{\bx-\by}=\norm{\bE}D,$$
We define  $G\equiv\max_{\bx\in X}\norm{\nabla g(\bE \bx)}$ to be the maximal norm of the gradient of $g$ over $\bE X$.

Problem~\eqref{eq:Problem} possesses some properties, which we present in the following lemmas.

\begin{lemma}[Lemma 14,\cite{WangLin2014}]\label{lemma:optimalIsSubspace}
	Let $X^*$ be the optimal set of problem \eqref{eq:Problem}. Then, there exists a constant vector $\bt^*$ and a scalar $s^*$ such that any optimal solution $\bx^*\in X^*$ satisfies $\bE\bx^*=\bt^*$ and $\inprod{\bb}{\bx^*}=s^*$.
\end{lemma}
Although the proof of the lemma in the given reference is for polyhedral sets, the extension for any convex set is trivial.

\begin{lemma}\label{lemma:OFUpperBound}
Let $f^*$ be the optimal value of problem~\eqref{eq:Problem}. Then, for any $\bx\in X$
$$f(\bx)-f^*\leq {C}$$
where $C=GD_\bE+\norm{\bb}D$.
\end{lemma}
\begin{proof}
Let $\bx^*$ be some optimal solution of  problem~\eqref{eq:Problem}, so that $f(\bx^*)=f^*$. Then for any $\bx\in X$, it follows from the convexity of $f$ that
\begin{equation*}
\begin{aligned}
f(\bx)-f(\bx^*)&\leq \inprod{\nabla f(\bx)}{\bx-\bx^*}\\
&=\inprod{\nabla g(\bE\bx)}{\bE\bx-\bE\bx^*}+\inprod{\bb}{\bx-\bx^*}\\
&\leq \norm{\nabla g(\bE\bx)}\norm{\bE\bx-\bE\bx^*}+\norm{\bb}\norm{\bx-\bx^*}\\
&\leq GD_\bE+\norm{\bb}D= C
\end{aligned}
\end{equation*}
where the last two inequalities are due to the Cauchy-Schwartz inequality and the definition of $G$,$D$ and $D_\bE$.
\end{proof}

The following lemma provides an \emph{error bound}, i.e., a bound on the distance of any feasible solution to the optimal set. This error bound will later be used as an alternative to a strong convexity assumption on $f$, which is usually needed in order to prove a linear rate of convergence. This is a different bound than the one given in \cite{WangLin2014}, since it relies heavily on the compactness of the set $X$, thus enabling to circumvent the use of the so-called gradient mapping.
\begin{lemma}\label{lemma:ErrBound}
For any $\bx\in X$,
$$d(\bx,X^*)^2\leq \kappa (f(\bx)-f^*),$$
where $\kappa=\theta^2 \left(\norm{\bb}D+3GD_\bE+\frac{2(G^2+1)}{\sigma_g}\right)$, and $\theta$ is the Hoffman constant associated with matrix $\left[\bA^T,\bE^T,\bb\right]^T$.
\end{lemma}
\begin{proof}
Lemma~\ref{lemma:optimalIsSubspace} implies that the optimal solution set $X^*$ can be defined as $X^*=X\cap S$ where $S=\paren{\bx\in \Real^n: \bE\bx=\bt^*,\;\inprod{\bb}{\bx}=s^*}$ for some $\bt^* \in \Real^m$ and $s^* \in \Real$. For any
$\bx\in X$, applying Lemma~\ref{lemma:Hoffman} with $\tilde{\bE}=\left[{\bE}^T,\bb\right]^T$, we have that
\begin{align}\label{eq:HoffmanForGivenProblem}
d(\bx,X^*)^2&\leq \theta^2 (\left(\inprod{\bb}{\bx}-s^*\right)^2+\norm{\bE\bx-\bt^*}^2),
\end{align}
where $\theta$ is the Hoffman constant associated with matrix $\left[\bA^T,\bE^T,\bb\right]^T$. Now, let $\bx \in X$ and $\bx^* \in X^*$.
Utilizing the $\sigma_g$-strong convexity of $g$, it follows that
\begin{equation}\label{eq:StrongConvexgEx}
\inprod{\nabla g(\bE\bx^*)}{\bE\bx-\bE\bx^*}+\frac{\sigma_g}{2}\norm{\bE\bx-\bE\bx^*}^2
\leq g(\bE\bx)-g(\bE\bx^*).\end{equation}
By the first order optimality conditions for problem~\eqref{eq:Problem}, we have (recalling that $\bx\in X$ and $\bx^*\in X^*$)
\begin{equation}\label{eq:FirstOrderOptimalityCond}
\inprod{\nabla f(\bx^*)}{\bx-\bx^*}\geq 0.\end{equation} Therefore,
\begin{equation}\label{eq:UBonStrongConvexPart0}
\begin{aligned}[b]
\frac{\sigma_g}{2}\norm{\bE\bx-\bt^*}^2&\leq \inprod{\nabla f(\bx^*)}{\bx-\bx^*}+\frac{\sigma_g}{2}\norm{\bE\bx-\bE\bx^*}^2\\
&=\inprod{\nabla g(\bE\bx^*)}{\bE\bx-\bE\bx^*}+\inprod{\bb}{\bx-\bx^*}+\frac{\sigma_g}{2}\norm{\bE\bx-\bE\bx^*}^2
\end{aligned}
\end{equation}
Now, using \eqref{eq:StrongConvexgEx} we can continue \eqref{eq:UBonStrongConvexPart0} to obtain
\begin{equation}\label{eq:UBonStrongConvexPart}
\begin{aligned}
\frac{\sigma_g}{2}\norm{\bE\bx-\bt^*}^2&\leq g(\bE\bx)-g(\bE\bx^*)+\inprod{\bb}{\bx}-\inprod{\bb}{\bx^*}=f(\bx)-f(\bx^*).\\
\end{aligned}
\end{equation}

We are left with the task of upper bounding $(\inprod{\bb}{\bx}-s^*)^2$. By the definitions of $s^*$ and $f$ we have that
\begin{equation}\label{eq:OptimalityConditions}
\begin{aligned}[b]
\inprod{\bb}{\bx}-s^*&=\inprod{\bb}{\bx-\bx^*}\\
&=\inprod{\nabla f(\bx^*)}{\bx-\bx^*}-\inprod{\nabla g(\bE\bx^*)}{\bE\bx-\bE\bx^*}\\
&=\inprod{\nabla f(\bx^*)}{\bx-\bx^*}- \inprod{\nabla g(\bt^*)}{\bE\bx-\bt^*}.
\end{aligned}\end{equation}
Therefore, using \eqref{eq:FirstOrderOptimalityCond}, \eqref{eq:OptimalityConditions} as well as the Cauchy-Schwartz inequality, we can conclude the following:
\begin{equation}\label{eq:UBonLinear0a}
s^*-\inprod{\bb}{\bx}\leq \inprod{\nabla g(\bt^*)}{\bE\bx-\bt^*}\leq \norm{\nabla g(\bt^*)}\norm{\bE\bx-\bt^*}.
\end{equation}
On the other hand, exploiting \eqref{eq:OptimalityConditions}, the convexity of $f$ and the Cauchy-Schwartz inequality, we also have that
\begin{equation}\label{eq:UBonLinear0b}
\begin{aligned}[b]
\inprod{\bb}{\bx}-s^*&=
\inprod{\nabla f(\bx^*)}{\bx-\bx^*}- \inprod{\nabla g(\bt^*)}{\bE\bx-\bt^*}\\
&\leq f(\bx)-f^*-\inprod{\nabla g(\bt^*)}{\bE\bx-\bt^*}\\
&\leq f(\bx)-f^*+\norm{\nabla g(\bt^*)}\norm{\bE\bx-\bt^*}.
\end{aligned}
\end{equation}
Combining \eqref{eq:UBonLinear0a}, \eqref{eq:UBonLinear0b}, and the fact that $f(\bx)-f^*\geq 0$, we obtain that
\begin{equation}\label{eq:UBonLinear1}
\begin{aligned}
(\inprod{\bb}{\bx}-s^*)^2&\leq \left(f(\bx)-f^*+\norm{\nabla g(\bt^*)}\norm{\bE\bx-\bt^*}\right)^2.\\
\end{aligned}
\end{equation}

Moreover, the definitions of $G$ and $D_\bE$ imply $\norm{\nabla g(\bt^*)}\leq G$, $\norm{\bE\bx-\bt^*}\leq D_\bE$, and since $\bx\in X$, it follows from Lemma~\ref{lemma:OFUpperBound} that $f(\bx)-f^*\leq C=GD_\bE+\norm{\bb}D$. Utilizing these bounds, as well as \eqref{eq:UBonStrongConvexPart} to bound \eqref{eq:UBonLinear1} results in
\begin{equation}\label{eq:UBOnLinearPart2}
\begin{aligned}[b]
(\inprod{\bb}{\bx}-s^*)^2&\leq \left(f(\bx)-f^*+G\norm{\bE\bx-\bt^*}\right)^2\\
&= (f(\bx)-f^*)^2+2G\norm{\bE\bx-\bt^*}(f(\bx)-f^*)+G^2\norm{\bE\bx-\bt^*}^2\\
&\leq (f(\bx)-f^*)C+2GD_\bE(f(\bx)-f^*)+G^2\frac{2}{\sigma_g}(f(\bx)-f^*)\\
&= (f(\bx)-f^*)\left(C+2GD_\bE+\frac{2G^2}{\sigma_g}\right)\\
&=(f(\bx)-f^*)\left(\norm{\bb}D+3GD_\bE+\frac{2G^2}{\sigma_g}\right).
\end{aligned}
\end{equation}
Plugging \eqref{eq:UBonStrongConvexPart} and \eqref{eq:UBOnLinearPart2} back into \eqref{eq:HoffmanForGivenProblem}, we obtain the desired result:
\begin{align*}
d(\bx,X^*)^2&\leq \theta^2 \left(\norm{\bb}D+3GD_\bE+\frac{2(G^2+1)}{\sigma_g}\right)(f(\bx)-f^*).
\end{align*}
\end{proof}

\subsection{Conditional Gradient and Linear Oracles} \label{sec:CG}

In order to present the CG algorithm, we first define the concept of linear oracles.

\begin{defin}[Linear Oracle]
	Given a set $X$, an operator $\mathcal{O}_X:\Real^n\rightarrow X$ is called a {\bf linear oracle} for $X$, if for each $\bc\in \Real^n$ it returns a vector $\bp\in X$ such that $\inprod{\bc}{\bp}\leq \inprod{\bc}{\bx}$ for any $\bx\in X$, i.e., $\bp$ is a minimizer of the linear function $\inprod{\bc}{\bx}$ over $X$.
\end{defin}

Linear oracles are black-box type functions, where the actual algorithm used in order to obtain the minimizer is unknown. For many feasible sets, such as $\ell_p$ balls and specific polyhedral sets, the oracle can be represented by a closed form solution or can be computed by an efficient method.

The CG algorithm and its variants are linear-oracle based algorithms.
The original CG algorithm, presented in \cite{FrankWolfe1956} -- also known as the Frank-Wolfe algorithm --  is as follows.
\begin{framed}\label{alg:CG}
	\setstretch{1.2}
	\noindent\underline{\bf Conditional Gradient Algorithm (CG)}\\	
	Input: A linear oracle ${\mathcal{O}}_{X}$\\	
	Initialize: $\bx^1\in X$ \newline	
	For $k=1,2,\ldots$
	\vspace{-\topsep}
	\begin{enumerate}	
		\itemsep0em 	
		\item Compute $\bp^k:={\mathcal{O}}_{X}(\nabla f(\bx^{k}))$.
		\item Choose a stepsize $\gamma^k$.
		\item Update $\bx^{k+1}:=\bx^{k}+\gamma^k(\bp^k-\bx^k)$.	
	\end{enumerate}
	\vspace{-\topsep}
\end{framed}
The algorithm is guaranteed to have an $O(\tfrac{1}{k})$ rate of convergence for stepsize determined according to exact line search \cite{FrankWolfe1956}, adaptive stepsize \cite{levitin1966} and predetermined stepsize \cite{Dunn1978}.
This upper bound on the rate of convergence is tight \cite{CanonCulum1968} and therefore variants, such as the ASCG were developed.

\section{Away Steps Conditional Gradient}\label{sec:AwayStepConditionalGradient}
The ASCG algorithm was proposed by Frank-Wolfe in \cite{Wolfe1970}. A linear convergence rate was proven for problems consisting of minimizing strongly convex objective functions over polyhedral feasible sets in \cite{Guelat1986} under some restrictions on the location of the optimal solution, and in \cite{Jaggi2013} without such restrictions. Jaggi and Lacoste-Julien \cite{Jaggi2013} showed that the latter result is also applicable for the specific case of problem~\eqref{eq:Problem} where $\bb=\bzero$ (or more generally $\bb\in \image(\bE)$), provided that an appropriate linear-oracle is available for the set $\bE X$. In this section, we extend this result for the general case of problem~\eqref{eq:Problem}, i.e., for any $\bE$ and $\bb$.  Furthermore, we explore the potential issues with obtaining a linear-oracle for the set $\bE X$, and suggest an alternative analysis, which only assumes existence of an appropriate linear-oracle on the original set $X$. Moreover, our analysis differs from the one presented in \cite{Jaggi2013} by the fact that it is based on duality rather than geometric arguments. This approach enables to derive a computable constant for the rate of convergence, which is explicitly expressed as a function of the problem's parameters and the geometry of the feasible set.

We separate the discussion of the ASCG into four sections. In Section~\ref{sec:VertexLinearOracles} we define the concept of \emph{vertex linear oracles}, which is needed for the ASCG method, and the issues of obtaining such an oracle for linear transformations of simple sets. Section~\ref{sec:ASCGmethod} contains a full description of the ASCG method itself, including the concept of vertex representation, and representation reduction. In Section~\ref{sec:RateConvergenceAnalysis} we present the rate of convergence analysis of the ASCG for problem~\eqref{eq:Problem}, as well as introduce the new computable convergence constant $\Omega_X$. Finally, in Section~\ref{sec:FindingOmegaForPolyhedrons} we demonstrate how to compute $\Omega_X$ for three types of simple sets.

\subsection{Vertex Linear Oracles}\label{sec:VertexLinearOracles}
The ASCG algorithm requires a linear oracle which is a \emph{vertex linear oracle}, a concept that we now define explicitly.
\begin{defin}[Vertex Linear Oracle]
	Given a polyhedral set $X$ with vertex set $V$, a linear oracle $\tilde{\mathcal{O}}_X:\Real^n\rightarrow V$ is called a {\bf vertex linear oracle} for $X$, if for each $\bc\in \Real^n$ it returns a vertex $\bp\in V$ such that $\inprod{\bc}{\bp}\leq \inprod{\bc}{\bx}$ for any $\bx\in X$.
\end{defin}
Notice that, according to the fundamental theorem of linear programming \cite[Theorem 2.7]{Bertsimas1997}, the problem of optimizing any linear objective function over the compact set $X$ always has an optimal solution which is a vertex. Therefore, the vertex linear oracle $\tilde{\mathcal{O}}_X$ is well defined. We also note that in this paper the term ``vertex" is synonymous with the term ``extreme point"

In \cite{Jaggi2013}, Jaggi and Lacoste-Julien proved that the ASCG algorithm is affine invariant. This means that given the problem \begin{equation}\label{eq:JaggiAffineInvariant}
\min_{\bx\in X}g(\bE\bx),
\end{equation} where $g$ is a strongly convex function and $\bE$ is some matrix, applying the ASCG algorithm on the equivalent problem \begin{equation}\label{eq:JaggiAffineInvariant2}\min_{\by\in Y} g(\by),\end{equation} where $Y=\bE X$, yields a linear rate of convergence, which depends only on the strong convexity parameter of $g$ and the geometry of the set $Y$ (regardless of what $\bE$ generated it). However, assuming that $\bE$ is not of a full column rank, i.e., $f$ is not strongly convex, retrieving an optimal solution $\bx^*\in X$ from the optimal solution $\by^*\in Y$ requires solving a linear feasibility problem. This feasibility problem is equivalent to solving the following constrained least squares problem:
$$\min_{\bx\in X}\norm{\bE\bx-\by^*}^2,$$ which, for a general $\bE$, may be more computationally expensive than simply applying the linear oracle on set $X$.  Moreover, in order to apply the algorithm to problem~\eqref{eq:JaggiAffineInvariant2}, a vertex linear oracle must be available for the set $Y=\bE X$. Assuming there exists a vertex linear oracle $\tilde{\mathcal{O}}_X$ for $X$, constructing such an oracle $\tilde{\mathcal{O}}_{\bE X}$ for $\bE X$ may incur an additional computational cost per iteration. A naive approach to construct a general linear oracle ${\mathcal{O}}_{\bE X}$, given  $\tilde{\mathcal{O}}_X$, is by the formula
\begin{equation}\label{eq:naive}{\mathcal{O}}_{\bE X}(\bc)=\bE\tilde{\mathcal{O}}_{X}(\bE^T\bc).
\end{equation} However, the output $\tilde{\bp}={\mathcal{O}}_{\bE X}(\bc)$ of this linear oracle is not guaranteed to  be a vertex of $\bE X$, and therefore, in order to obtain a vertex linear oracle $\tilde{\mathcal{O}}_{\bE X}(\bc)$, a vertex $\bp$ of $\bE X$ with the same objective function value as $\tilde{\bp}$ must still be found. As an example, take $X$ to be the unit box in three dimensions, $X=[-1,1]^3\subseteq \Real^3$, and let $\bE$ be given by
$$\bE=\begin{bmatrix}1& 1& 1\\ 1& 1& -1\\ 0& 0& 2\end{bmatrix}.$$
We denote the vertex set $V$ of the set $X$ by the letters A-H as follows:
\begin{align*}
A&=(1,1,1)^T,\; &B&=(1,1,-1)^T,\; &C&=(1,-1,-1)^T,\; &D&=(1,-1,1)^T,\\
E&=(-1,1,1)^T,\; &F&=(-1,-1,1)^T,\; &G&=(-1,1,-1)^T,\; &H&=(-1,-1,-1)^T,
\end{align*}
and the linear mappings of these vertices by the matrix $\bE$ by A'-H':
\begin{align*}
A'&=(3,1,2)^T,\; &B'&=(1,3,-2)^T,\; &C'&=G'=(-1,1,-2)^T,\\
F'&=(-1,-3,2)^T,\; &H'&=(-3,-1,-2)^T,\; &D'&=E'=(1,-1,2)^T.
\end{align*}
The vertex set of $\bE X$ is $\ext(\bE X)=\{A',B',F',H'\}$.

\begin{figure}[t]\label{fig:X and EX}
\caption{The sets $X$ and $\bE X$}
\includegraphics{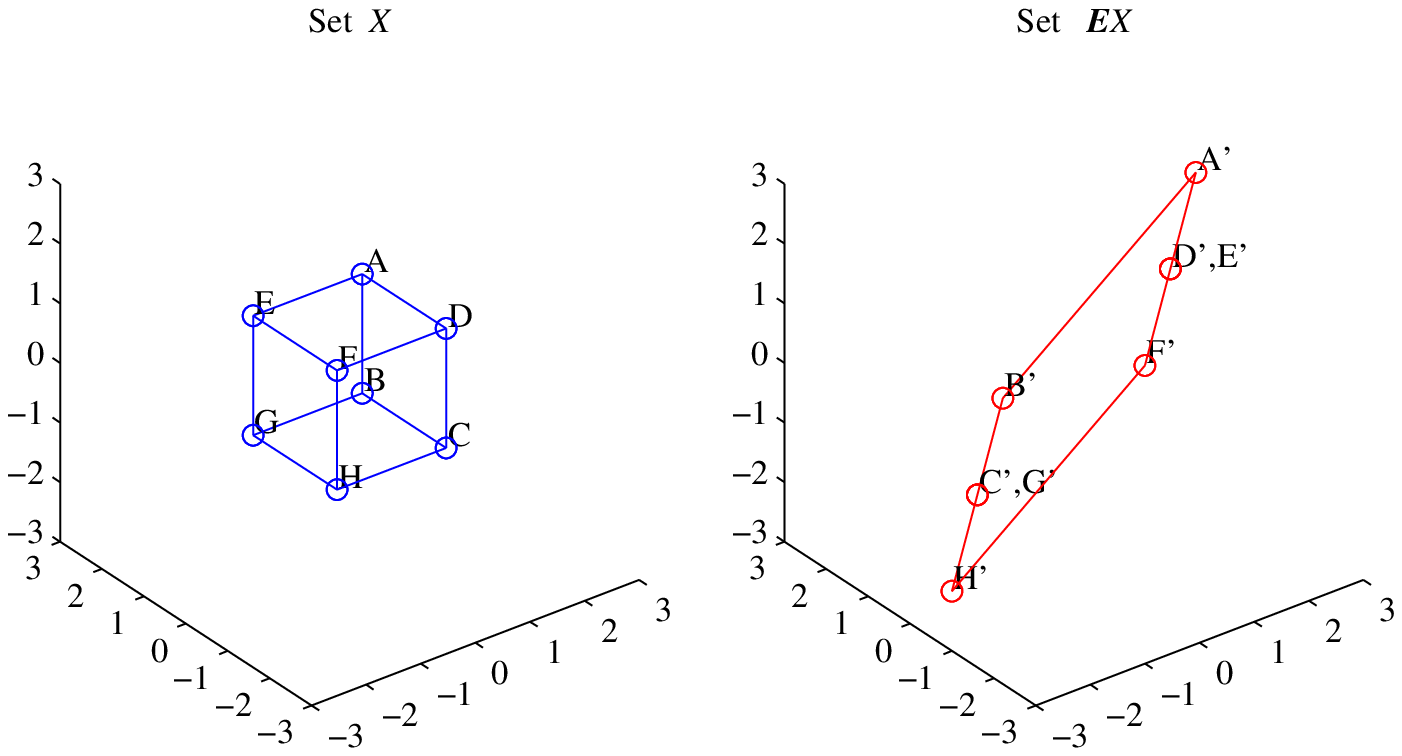}
\end{figure}

The sets $X$ and $\bE X$ are presented in Figure~\ref{fig:X and EX}. 
Notice that finding a vertex linear oracle for $X$ is trivial, while finding one for $\bE X$ is not. In particular, a vertex linear oracle for $X$ may be given by any operator $\tilde{\mathcal{O}}_X(\cdot)$ satisfying
\begin{equation}\label{eq:VertexLinearOracleForBox}
\tilde{\mathcal{O}}_X(\bc)\in\argmin_{\bx\in V}\paren{\inprod{\bc}{\bx}}=\paren{\bx\in \{-1,1\}^3:x_ic_i=-|c_i|,\,\forall i=1,\ldots,n},\quad \forall\; \bc \in \Real^3.
\end{equation}  Given the vector $\bc=(-1,\;1,\;3)^T$, we want to find $$\bp\in \argmin_{\by\in \ext(\bE X)} \inprod{\bc}{\by}.$$ Using the naive approach, described in \eqref{eq:naive}, we obtain a vertex of $X$ by applying the vertex linear oracle $\tilde{\mathcal{O}}_X$ described in \eqref{eq:VertexLinearOracleForBox} with parameter $\bE^T\bc=(0,\; 0 ,\; 1)$, which may return either one of the vertices B, C, G or H. If vertex C is returned, then its mapping C' does not yield a vertex in $\bE X$. Therefore, the oracle $\tilde{\mathcal{O}}_{\bE X}$ must now search for a vertex with the same objective function value, or alternatively, discover that C' lies on the face defined by B' and H', and consequently return one of these vertices. Obviously, this is true for any $\bc$ such that $\tilde{\mathcal{O}}_X(\bE^T\bc)$ returns one of the vertices C, D, E or G. This 3D example illustrates that, even for a simple $X$, understanding the geometry of the set $\bE X$, let alone constructing a vertex linear oracle over it, is not trivial and becomes more complicated as the dimension of the problem increases.

We aim to show that given a vertex linear oracle for $X$, the ASCG algorithm converges in a linear rate for problem~\eqref{eq:Problem}. Since in our analysis we do not assume the existence of a vertex linear oracle for $\bE X$, but rather a vertex linear oracle for $X$, the computational cost per iteration is independent of the matrix $\bE$, and depends only on the geometry of $X$.

\subsection{The ASCG Method}\label{sec:ASCGmethod}
We will now present the ASCG algorithm. In the following we denote the vertex set of $X$ as $V=\ext(X)$. Moreover, as part of the ASCG algorithm, at each iteration $k$ the iterate $\bx^k$ is represented as a convex combination of points in $V$. Specifically, $\bx^k$ is assumed to have the representation $$\bx^k=\sum_{\bv\in V} \mu^k_\bv\bv,$$ where $\bmu^k\in \Delta_{|V|}$.
Let $U^k=\paren{\bv\in V:\mu^k_\bv>0}$, then $U^k$ and $\paren{\mu^k_{\bv}}_{\bv\in U^k}$ provide a compact representation of $\bx^k$, and $\bx^k$ lies in the relative interior of the set $\conv(U^k)$. Throughout the algorithm we update $U^k$ and $\bmu^k$ via the vertex representation updating (VRU) scheme. The ASCG method has two types of updates: a \emph{forward step}, used in the classical CG algorithm, where a vertex is added to the representation, and an \emph{away step}, unique to this algorithm, in which the coefficient of one of the vertices used in the representation is reduced or even nullified. Specifically, the away step uses the direction $(\bx^k-\bu^k)$ where $\bu^k\in U^k$ and step size $\gamma^k>0$ so that
\begin{equation*}
\begin{aligned}
\bx^{k+1}&=\bx^k+\gamma^k(\bx^k-\bu^k)\\
&=(\bx^k-\mu^k_{\bu^k}\bu^k)(1+\gamma^k)+(\mu^k_{\bu^k}-\gamma^k(1-\mu^k_{\bu^k}))\bu^k\\
&=\sum_{\bv\in U^k/\paren{\bu^k}}(1+\gamma^k)\mu_{\bv}\bv+(\mu^k_{\bu^k}(1+\gamma^k)-\gamma^k)\bu^k,\\
\end{aligned}
\end{equation*}
and so $\mu^{k+1}_{\bu^k}=\mu^k_{\bu^k}-\gamma^k(1-\mu^k_{\bu^k})<\mu^k_{\bu^k}$. Moreover, if $\gamma^k=\frac{\mu^k_{\bu^k}}{1-\mu^k_{\bu^k}}$, then $\mu^{k+1}_{\bu^k}$ is nullified, and consequently, the vertex $\bu^k$ is removed from the representation. This vertex removal is referred to as a \emph{drop step}.

The full description of the ASCG algorithm and the VRU scheme is given as follows.
\begin{framed}\label{alg:ASCG}
	\setstretch{1.2}
	\noindent\underline{\bf Away Step Conditional Gradient algorithm (ASCG)}\\	
	Input: A vertex linear oracle $\tilde{\mathcal{O}}_{X}$\\	
	Initialize: $\bx^1\in V$ where $\mu^1_{\bx^1}=1$, $\mu^1_{\bv}=0$ for any $\bv\in V/\paren{\bx^1}$ and $U^1=\{\bx^1\}$\newline	
	For $k=1,2,\ldots$
	\vspace{-\topsep}
	\begin{enumerate}	
		\itemsep0em 	
		\item Compute $\bp^k:=\tilde{\mathcal{O}}_{X}(\nabla f(\bx^{k}))$.
		\item Compute $\bu^k\in \argmax\limits_{\bv\in U^k} \inprod{\nabla f(\bx^{k})}{\bv}$.
		\item If $\inprod{\nabla f(\bx^{k})}{\bp^k-\bx^k}\leq  \inprod{\nabla f(\bx^{k})}{\bx^k-\bu^k}$, then set $\bd^k:=\bp^k-\bx^k$ and $\overline{\gamma}^k:=1$.\\		
		Otherwise, set $\bd^k:=\bx^k-\bu^k$ and $\overline{\gamma}^k:=\frac{\mu^k_{\bu^k}}{1-\mu^k_{\bu^k}}$
		\item Choose a stepsize $\gamma^k$.
		\item Update $\bx^{k+1}:=\bx^{k}+\gamma^k\bd^k$.
\item Employ the VRU procedure with input $(\bx^k,U^k,\bmu^k,\bd^k,\gamma^k,\bp^k,\bv^k)$ and obtain an updated representation $(U^{k+1},\bmu^{k+1})$.
	\end{enumerate}
	\vspace{-\topsep}
\end{framed}
The stepsize in the ASCG algorithm can be chosen according to one of the following stepsize selection rules, where $\bd^k$ and $\overline{\gamma}^k$ are as defined in the algorithm.
\begin{equation}\label{eq:StepSize}
\gamma^k\begin{cases}
		\in \argmin\limits_{0\leq \gamma\leq \overline{\gamma}^k} f(\bx^k+\gamma\bd^k) & \text{Exact line search}\\
		=\min\paren{\!-\frac{\inprod{\nabla f(\bx^k)}{\bd^k}}{\rho\norm{\bd^k}^2},\overline{\gamma}^k}\!\in \!\argmin\limits_{0\leq \gamma\leq \overline{\gamma}^k} \left\{\!\gamma\inprod{\nabla f(\bx^k)}{\bd^k} +{\gamma}^2\frac{\rho}{2}\norm{\bd^k}^2\right\} &
		\text{Adaptive \cite{levitin1966}.}
		\end{cases}
\end{equation}
\begin{remark} \label{remark:monotone}
It is simple to show that under the above two choice of stepsize strategies, the sequence of function values $\{f(\bx^k)\}_{k \geq 1}$ is nonincreasing.
\end{remark}
Since the convergence rate analyses for both of these stepsize options is similar, we chose to conduct a unified analysis for both cases. Following is  exact definition of the VRU procedure.

\begin{framed}\label{alg:UpdatingSchemeForUkmuk}
	\setstretch{1.2}
	\noindent\underline{\bf Vertex Representation Updating (VRU) Procedure}\\
	{\bf Input:} $\bx^k$  - current point. \\ $(U^k,\bmu^k)$  -  vertex representation of $\bx^k$,\\
 $\bd^k,\gamma^k$ - current direction and stepsize,\\  $\bp^k, \bv^k$ - candidate vertices. \\
	{\bf Output:} Updated vertex representation $(U^{k+1},\bmu^{k+1})$ of $\bx^{k+1}=\bx^k+\gamma^k\bd^k$.\\
	If $\bd^k=\bx^k-\bu^k$ (away step) then
	\vspace{-\topsep}
	\begin{enumerate}
		\itemsep0em
		\item Update $\mu^{k+1}_{\bv}:=\mu^{k}_{\bv}(1+\gamma^k)$ for any ${\bv}\in U^{k}/\paren{\bu^k}$.
		\item Update $\mu^{k+1}_{\bu^k}:=\mu^{k}_{\bu^k}(1+\gamma^k)-\gamma^k$.
		\item If $\mu^{k+1}_{\bu^k}=0$ (drop step), then update $U^{k+1}:=U^k/\paren{\bu^k}$, otherwise $U^{k+1}:=U^k$.
	\end{enumerate}
	\vspace{-\topsep}
	Else ($\bd^k=\bp^k-\bx^k$ - forward step)
	\vspace{-\topsep}
	\begin{enumerate}
		\itemsep0em
		\item Update $\mu^{k+1}_{\bv}:=\mu^{k}_{\bv}(1-\gamma^k)$ for any ${\bv}\in U^k/\paren{\bp^k}$.
		\item Update $\mu^{k+1}_{\bp^k}:=\mu^{k}_{\bp^k}(1-\gamma^k)+\gamma^k$.
		\item If $\mu^{k+1}_{\bp^k}=1$, then update $U^{k+1}=\paren{\bp^k}$, otherwise update $U^{k+1}:=U^k\cup \paren{\bp^k}$. 		
	\end{enumerate}	
	\vspace{-\topsep}
	Update $(U^{k+1},\bmu^{k+1}):=\mathcal{R}(U^{k+1},\bmu^{k+1})$ with $\mathcal{R}$ being a representation reduction procedure with constant $N$.
\end{framed}
\noindent The VRU scheme uses a representation reduction procedure $\mathcal{R}$ with constant $N$, which is a procedure that takes a representation $(U,\bmu)$ of a point $\bx$ and replaces it by a representation $(\tilde{U},\tilde{\bmu})$ of $\bx$ such that $\tilde{U}\subseteq U$ and $|\tilde{U}|\leq N$. We consider two possible options for the representation reduction procedure:
 \begin{enumerate} \item $\mathcal{R}$ is the trivial procedure, meaning it does not change the representation, in which case its constant is $N=|V|$.
  \item The procedure $\mathcal{R}$ is some implementation of the Carath\'{e}odory theorem \cite[Section 17]{Rockafellar1970}, in which case its constant is $N=n+1$. Using this option will accelerate the algorithm when the number of vertices is not polynomial in the problem's dimension. A full description of the incremental representation reduction (IRR) scheme, which applies the Carath\'{e}odory theorem efficiently in this context, is presented in Appendix~\ref{appx:Caratheodory}.
\end{enumerate}
\subsection{Rate of Convergence Analysis}\label{sec:RateConvergenceAnalysis}
We will now prove the linear rate of convergence for the ASCG algorithm for problem \eqref{eq:Problem}.
In the following we use $I(\bx)$ to denote the \emph{index set of the active constraints at $\bx$},
$$I(\bx)=\paren{i\in\paren{1,\ldots,n}: \bA_i\bx=a_i}.$$
Similarly, for a given set $U$, the set of active constraints for all the points in $U$ is defined as
$$I(U)=\paren{i\in\paren{1,\ldots,n}: \bA_i\bv=a_i,\;\forall \bv\in U}=\bigcap_{\bv\in U} I(\bv).$$

We present the following technical lemma, which is similar to a result presented by
Jaggi and Lacoste-Julien \cite{Jaggi2013}\footnote{This was done as part of the proof of \cite[Lemma 6]{Jaggi2013}, and does not appear as a separate lemma.}. In \cite{Jaggi2013} the proof is based on geometrical considerations, and utilizes the so-called  ``pyramidal width constant", which is the optimal value of a complicated optimization problem, whose value is unknown even for simple sets such as the unit simplex. In contrast, the proof below relies on simple linear programming duality arguments, and in addition, the derived constant $\Omega_X$, which replaces the pyramidal width constant, is computable for a many choices of sets $X$.
\begin{lemma}\label{lemma:PyrWidth2}
Given $U\subseteq V$ and $\bc\in\Real^n$. If there exists a $\bz\in\Real^n$ such that $\bA_{I(U)}\bz\leq 0$ and $\inprod{\bc}{\bz}>0$,
then
$$\max_{\bp\in V, \bu\in U}\inprod{\bc}{\bp-\bu}\geq \frac{\Omega_X}{|U|}\frac{\inprod{\bc}{\bz}}{\norm{\bz}}$$
where
\begin{equation} \label{defOmega} \Omega_X=\frac{\zeta}{\varphi}\end{equation}
for
\begin{equation*}
\begin{aligned}
\zeta&=\min\limits_{\bv\in V,i\in\paren{1,\ldots,m}:a_i>\bA_i\bv}(a_i-\bA_i\bv),\\
\varphi&=\max\limits_{i\in \paren{1,\ldots,m}/I(V)} \norm{\bA_i}.
\end{aligned}
\end{equation*}
\end{lemma}
\begin{proof}
By the fundamental theorem of linear programming \cite{Goldfarb1989}, we can  maximize the function $\inprod{\bc}{\bx}$ on $X$ instead of on $V$ and get the same optimal value. Similarly, we can minimize the function $\inprod{\bc}{\by}$ on $\conv(U)$ instead of on $U$, and obtain the same optimal value. Therefore,
\begin{equation}\label{eq:LemmaPyrWidth2_DiscreteToContLinearProgram}
\begin{aligned}[b]
\max_{\bp\in V, \bu\in U}\inprod{\bc}{\bp-\bu}&=\max_{\bp\in V}\inprod{\bc}{\bp}-\min_{\bu\in U}\inprod{\bc}{\bu}\\
&=\max_{\bx\in X} \inprod{\bc}{\bx}-\min_{\by\in \conv(U)}\inprod{\bc}{\by}\\
&=\max_{\bx:\bA\bx\leq\ba} \inprod{\bc}{\bx}+\max_{\by\in \conv(U)}\paren{-\inprod{\bc}{\by}}.
\end{aligned}
\end{equation}

Since $X$ is nonempty and bounded, the problem in $\bx$ is feasible and bounded above. Therefore, by strong duality for linear programming,
\begin{equation}\label{eq:LemmaPyrWidth2_Duality}
\max_{\bx:\bA\bx\leq\ba} \inprod{\bc}{\bx}=\min_{\beeta\in\Real^m_+: \bA^T\beeta=\bc} \inprod{\ba}{\beeta}.
\end{equation}
Plugging \eqref{eq:LemmaPyrWidth2_Duality} back into \eqref{eq:LemmaPyrWidth2_DiscreteToContLinearProgram} we obtain:
\begin{equation}\label{eq:LemmaPyrWidth2_MinMax}
\begin{aligned}[b]
\max_{\bp\in V, \bu\in U}\inprod{\bc}{\bp-\bu}&=\min_{\beeta\in\Real^m_+: \bA^T\beeta=\bc} \inprod{\ba}{\beeta}+\max_{\by\in \conv(U)}\paren{-\inprod{\bc}{\by}}\\
&=\min_{\beeta\in\Real^m_+: \bA^T\beeta=\bc}\max_{\by\in \conv(U)} \inprod{\ba-\bA\by}{\beeta}.
\end{aligned}
\end{equation}
Since $\overline{\by}=\frac{1}{|U|}\sum_{\bv\in U}\bv$ is in $\conv(U)$, we have that
$$\max_{\by\in \conv(U)} \inprod{\ba-\bA\by}{\beeta}\geq \inprod{\ba-\bA\overline{\by}}{\beeta}$$
for any value of $\beeta$, and therefore,
\begin{equation}\label{eq:LemmaPyrWidth2_MinMax2}
\min_{\beeta\in\Real^m_+: \bA^T\beeta=\bc}\max_{\by\in \conv(U)} \inprod{\ba-\bA\by}{\beeta}\geq \min_{\beeta\in\Real^m_+: \bA^T\beeta=\bc}\inprod{\ba-\bA\overline{\by}}{\beeta}.
\end{equation}

Using strong duality on the RHS of \eqref{eq:LemmaPyrWidth2_MinMax2}, we obtain that
\begin{equation}\label{eq:LemmaPyrWidth2_Duality2}
\begin{aligned}
\min_{\beeta\in\Real^m_+: \bA^T\beeta=\bc} \inprod{\ba-\bA\overline{\by}}{\beeta}&=\max_{\bx}\paren{\inprod{\bc}{\bx}: \bA\bx\leq (\ba-\bA\overline{\by})}.
\end{aligned}
\end{equation}
Denote $J=I(U)$ and $\overline{J}=\paren{1,\ldots,m}/J$. From the definition of $I(U)$, it follows that
\begin{equation} \label{539} \ba_J-\bA_J\bv=\bzero\end{equation}  for all $\bv\in U$, and that for any $i\in \overline{J}$ there exists at least one vertex $\bv\in U$ such that $a_i-\bA_i\bv>0$, and hence, $$a_i-\bA_i\bv\geq \min_{\bu\in V, j\in\paren{1,\ldots,m}: \ba_j>\bA_j\bu}(a_j-\bA_j\bu)=\zeta>0,$$
which in particular implies that
\begin{equation} \label{541} \sum_{\bv \in U} (a_i -\bA_i \bv) \geq \zeta>0.\end{equation}
Since $\overline{\by}\in \conv(U)$, we can conclude from (\ref{539}) and (\ref{541}) that
\begin{equation}\label{eq:LemmaPyrWidth2_ChoiceOfy}
\begin{aligned}
\ba_J-\bA_J\overline{\by} &=\bzero\\
\ba_{\overline{J}}-\bA_{\overline{J}}\overline{\by}&=\frac{1}{|U|}\sum_{\bv\in U}(\ba_{\overline{J}}-\bA_{\overline{J}}\bv)\geq \bone\frac{\zeta}{|U|}.
\end{aligned}
\end{equation}
Therefore, replacing the RHS of the set of inequalities $\bA\bx\leq (\ba-\bA\overline{\by})$ in \eqref{eq:LemmaPyrWidth2_Duality2} by the bounds given in \eqref{eq:LemmaPyrWidth2_ChoiceOfy}, we obtain that
\begin{equation}\label{eq:LemmaPyrWidth2_Duality3}
\begin{aligned}
\max_{\bx}\paren{\inprod{\bc}{\bx}: \bA\bx\leq (\ba-\bA\overline{\by})}
&\geq \max_{\bx}\paren{\inprod{\bc}{\bx}: \bA_J\bx\leq \bzero,\;\bA_{\overline{J}}\bx\leq \bone\frac{\zeta}{|U|}}.
\end{aligned}
\end{equation}

Combining \eqref{eq:LemmaPyrWidth2_MinMax},\eqref{eq:LemmaPyrWidth2_MinMax2}, \eqref{eq:LemmaPyrWidth2_Duality2} and \eqref{eq:LemmaPyrWidth2_Duality3} it follows that
\begin{equation}\label{eq:LemmaPyrWidth2_Final}
\max_{\bp\in V, \bu\in U}\inprod{\bc}{\bp-\bu}
\geq Z^*,
\end{equation}
where
\begin{equation}\label{eq:LemmaPyrWidth2_Final2}
Z^*=\max_{\bx}\paren{\inprod{\bc}{\bx}:\bA_J\bx\leq \bzero,\;\bA_{\overline{J}}\bx\leq \bone \frac{\zeta}{|U|}}.
\end{equation}
We will now show that it is not possible for $\bz$ to satisfy $\bA_{\overline{J}}\bz\leq 0$. Suppose by contradiction $\bz$ satisfies does satisfy $\bA_{\overline{J}}\bz\leq 0$. Then ${\bx}_{\alpha}=\alpha\bz$ is a feasible solution of problem~\eqref{eq:LemmaPyrWidth2_Final2} for any $\alpha>0$, and since $\inprod{\bc}{\bz}>0$ we obtain that $\inprod{\bc}{{\bx}_{\alpha}}\rightarrow\infty$ as $\alpha\rightarrow\infty$, and thus $Z^*=\infty$. However, since $V$ contains a finite number of points, the LHS of \eqref{eq:LemmaPyrWidth2_Final} is bounded from above, and so $Z^*<\infty$ in contradiction. Therefore, there exists $i\in \overline{J}$ such that $\bA_i\bz>0$.
Since $\bz\neq 0$, the vector $\overline{\bx}=\frac{\bz}{\norm{\bz}}\frac{\Omega_X}{|U|}$ is well defined. Moreover, $\overline{\bx}$ satisfies   \begin{equation}\label{eq:Feasible_xbar1}\bA_J\overline{\bx}=\frac{\Omega_X}{\norm{\bz}|U|}{\bA_J\bz}\leq 0,\end{equation} and
\begin{equation}\label{eq:Feasible_xbar2}
\begin{aligned}[b]
\bA_i\overline{\bx}=\bA_i\bz\frac{\Omega_X}{|U|\norm{\bz}}&\leq \norm{\bA_i}\norm{\bz}\frac{\zeta}{|U|\norm{\bz} \varphi  }\leq \frac{\zeta}{|U|},\quad \forall i\in \overline{J},
\end{aligned}
\end{equation}
where the first inequality follows from the Cauchy-Schwartz inequality and the second inequality follows from the fact that if $i\in\overline{J}$, then $i\notin I(V)$ and so $\norm{\bA_i}\leq\varphi$.
Consequently, \eqref{eq:Feasible_xbar1} and \eqref{eq:Feasible_xbar2} imply that $\overline{\bx}$ is a feasible solution for problem \eqref{eq:LemmaPyrWidth2_Final2}. Therefore, $Z^*\geq\inprod{\bc}{\overline{\bx}}$, which by \eqref{eq:LemmaPyrWidth2_Final} yields
\begin{equation*}\label{eq:LemmaPyrWidth2_Result}
\max_{\bp\in V, \bu\in U}\inprod{\bc}{\bp-\bu}
\geq \inprod{\bc}{\overline{\bx}}=\frac{\Omega_X}{|U|}\frac{\inprod{\bc}{\bz}}{\norm{\bz}}.
\end{equation*}
\end{proof}
The constant $\Omega_X$ represents a normalized minimal distance between the hyperplanes that contain facets of $X$ and the vertices of $X$ which do not lie on those hyperplanes. We will refer to $\Omega_X$ as \emph{the vertex-facet distance of $X$}. Examples for the derivation of $\Omega_X$ for some simple polyhedral sets can be found in Section~\ref{sec:FindingOmegaForPolyhedrons}.

The following lemma is a technical result stating that the active constraints at a given point are the same as the active constraints of the set of vertices in its compact representation.
\begin{lemma}\label{lemma:Ix_IU}
	Let $\bx\in X$ and the set $U\subseteq V$ satisfy $\bx=\sum_{\bv\in U} \mu_\bv\bv$, where $\bmu\in \Delta^+_{|U|}$. Then $I(\bx)=I(U)$.
\end{lemma}
\begin{proof}
It is trivially true that $I(U)\subseteq I(\bx)$ since $\bx$ is a convex combination of points in the affine space defined by $\paren{\by:\bA_{I(U)}\by=\ba_{I(U)}}$. We will prove that $I(\bx)\subseteq I(U)$. Any $\bv\in U\subseteq X$ satisfies $\bA_{I(\bx)}\bv\leq\ba_{I(\bx)}$. Assume to the contrary, that there exists $i\in I(\bx)$ such that some $\bu\in U$ satisfies  $\bA_i\bu<a_i$. Since $\mu_\bu>0$ and $\sum_{\bv\in U} \mu_\bv=1$, it follows that
	$$\bA_{i}\bx=\sum_{\bv\in U} \mu_\bv\bA_i\bv< \sum_{\bv\in U} \mu_\bv a_i =a_i,$$
	in contradiction to the assumption that $i\in I(\bx)$.	
\end{proof}
\begin{cor}\label{cor:boundOnInnerProd}
For any $\bx\in X/X^*$ which can be represented as $\bx=\sum_{\bv\in U} \mu_\bv\bv$ for some $\bmu\in \Delta^+_{|U|}$ and $U\subseteq V$, it holds that,
$$\max_{\bu\in U,\bp\in V} \inprod{\nabla f(\bx)}{\bu-\bp}\geq \frac{\Omega_X}{|U|}  \max_{\bx^*\in X^*}\frac{\inprod{\nabla f(\bx)}{\bx-\bx^*}}{\norm{\bx-\bx^*}}.$$
\end{cor}
\begin{proof}
For any $\bx\in X/X^*$ define $\bc=-\nabla f(\bx)$. It follows from Lemma~\ref{lemma:Ix_IU} that $I(U)=I(\bx)$.
For any $\bx^*\in X^*$, the vector $\bz=\bx^*-\bx$ satisfies  $$\bA_{I(U)}\bz=\bA_{I(\bx)}\bz=\bA_{I(\bx)}\bx^*-\bA_{I(\bx)}\bx\leq \ba_{I(\bx)}-\ba_{I(\bx)}=\bzero,$$ and, from the convexity of $f$, as well as the optimality of $\bx^*$, $\inprod{\bc}{\bz}=-\inprod{\nabla f(\bx)}{\bx^*-\bx}\geq f(\bx)-f(\bx^*)>0$. Therefore, invoking Lemma~\ref{lemma:PyrWidth2} achieves the desired result.
\end{proof}

We now present the main theorem of this section, which establishes the linear rate of convergence of ASCG for problem~\eqref{eq:Problem}. This theorem is an extension of \cite[Thorem 7]{Jaggi2013}, and the proof follows the same general arguments, while incorporating the use of the error bound from Lemma~\ref{lemma:ErrBound} and the new constant $\Omega_X$.
\begin{theorem}\label{thm:ConvergenceOfASCG}
Let $\{\bx^k\}_{k \geq 1}$ be the sequence generated by the ASCG algorithm for solving problem \eqref{eq:Problem} using a representation reduction to procedure $\mathcal{R}$ with constant $N$, and let $f^*$ be the optimal value of the problem. Then for any $k\geq 1$
\begin{equation} \label{613} f(\bx^k)-f^*\leq C(1-\alpha^{\dag})^{(k-1)/2},\end{equation}
where
\begin{equation} \label{defalpha} \alpha^\dag=\min \left \{ \frac{(\Omega_X)^2}{8\rho\kappa D^2N^2}, \frac{1}{2} \right \},\end{equation}
 $\kappa=\theta^2 \left(\norm{\bb}D+3GD_\bE+\frac{2(G^2+1)}{\sigma_g}\right)$ with
	$\theta$ being the Hoffman constant associated with matrix $\left[\bA^T,\bE^T,\bb\right]^T$, $C=GD_\bE+\norm{\bb}D$, and $\Omega_X$ is the vertex-facet distance of $X$ given in (\ref{defOmega}).
\end{theorem}
\begin{proof}
For each $k$ we will denote the stepsize generated by exact line search as $\gamma_e^k$ and the adaptive stepsize as $\gamma_a^k$. Then
	\begin{equation}\label{eq:RatioFunctionStepSizes2}
	\begin{aligned}		
		f(\bx^k+\gamma_e^k\bd^k)&\leq f(\bx^{k+1})\leq f(\bx^k+\gamma_a^k\bd^k).
	\end{aligned}
	\end{equation}	
From Lemma~\ref{lemma:DescentLemma} (the descent lemma), we have that
	\begin{equation}\label{eq:RatioFunctionStepSizes3}
	\begin{aligned}				
		 f(\bx^k+\gamma_a^k\bd^k)\leq f(\bx^k)+\gamma_a^k\inprodmy{\nabla f(\bx^k)}{\bd^k}+\frac{(\gamma_a^k)^2\rho}{2}\|\bd^k\|^2.
	\end{aligned}
	\end{equation}
Assuming that $\bx^k\notin X^*$, then for any $\bx^*\in X^*$ we have that
\begin{equation}\label{eq:BoundOnInnerProdWithdk}
\begin{aligned}[b]
\inprodmy{\nabla f(\bx^k)}{\bd^k}&=\min\paren{ \inprodmy{\nabla f(\bx^k)}{\bp^k-\bx^k}, \inprodmy{\nabla f(\bx^k)}{\bx^k-\bu^k}}\\
&\leq \inprodmy{\nabla f(\bx^k)}{\bp^k-\bx^k}\\
&\leq \inprodmy{\nabla f(\bx^k)}{\bx^*-\bx^k}\\
&\leq f^*-f(\bx^k),
\end{aligned}
\end{equation}
where the first equality is derived from the algorithm's specific choice of $\bd^k$, the third line follows from the fact that $\bp^k=\tilde{\mathcal{O}}_X(\nabla f(\bx^k))$, and the fourth line follows from the convexity of $f$.
In particular, $\bd^k\neq \bzero$, and by \eqref{eq:StepSize} it follows that $\gamma_a^k$ is equal to
\begin{equation}\label{eq:OptimalStepSize}
\gamma_a^k=\min\paren{-\frac{\inprodmy{\nabla f(\bx^k)}{\bd^k}}{\rho \normmy{\bd^k}^2},\overline{\gamma}^k}.\end{equation}
We now separate the analysis to three cases: (a) $\bd^k=\bp^k-\bx^k$ and $\gamma_a^k=\overline{\gamma}^k$, (b) $\bd^k=\bx^k-\bu^k$ and $\gamma_a^k=\overline{\gamma}^k$, and (c) $\gamma_a^k<\overline{\gamma}^k$.

In cases (a) and (b), it follows from \eqref{eq:OptimalStepSize} that
\begin{equation}\label{eq:CaseStepSizeMaximal}
\overline{\gamma}^k\rho\normmy{\bd^k}^2\leq -\inprodmy{\nabla f(\bx^k)}{\bd^k}.
\end{equation}
Using inequalities \eqref{eq:RatioFunctionStepSizes2}, \eqref{eq:RatioFunctionStepSizes3} and \eqref{eq:CaseStepSizeMaximal}, we obtain
\begin{equation*}\label{eq:RatioFunctionMaximalStepSizes0}
	\begin{aligned}		
		f(\bx^{k+1})&\leq f(\bx^k)+\gamma_a^k\inprodmy{\nabla f(\bx^k)}{\bd^k}+\frac{(\gamma_a^k)^2\rho}{2}\normmy{\bd^k}^2\\
		&\leq  f(\bx^k)+\frac{\overline{\gamma}^k}{2}\inprodmy{\nabla f(\bx^k)}{\bd^k}.
\end{aligned}
\end{equation*}
Subtracting $f^*$ from both sides of the inequality and using \eqref{eq:BoundOnInnerProdWithdk}, we have that
	\begin{equation}\label{eq:RatioFunctionMaximalStepSizes}
	\begin{aligned}[b]	
		f(\bx^{k+1})-f^*
		&\leq  f(\bx^k)-f^*+\frac{\overline{\gamma}^k}{2}\inprodmy{\nabla f(\bx^k)}{\bd^k}\\
		&\leq (f(\bx^k)-f^*)\left(1-\frac{\overline{\gamma}^k}{2}\right).
	\end{aligned}
	\end{equation}
In case (a), $\overline{\gamma}^k=1$, and hence
\begin{equation} \label{666} f(\bx^{k+1})-f^*\leq  \frac{f(\bx^k)-f^*}{2}.\end{equation}
In case (b), we have no positive lower bound on $\overline{\gamma}^k$, and therefore we can only conclude, by the nonnegativity of $\overline{\gamma}^k$, that
$$f(\bx^{k+1})-f^* \leq  {f(\bx^k)-f^*}.$$
However, case (b) is a drop step, meaning in particular that $|U^{k+1}|\leq |U^{k}|-1$, since before applying the representation reduction procedure $\mathcal{R}$, we eliminate one of the vertices in the representation of $\bx^{k}$. Denoting the number of drop steps until iteration $k$ as $s^k$, and the number of forward steps until iteration $k$ as $l^k$, it follows from the algorithm's definition that $l^k+s^k\leq k-1$ (at each iteration we add a vertex, remove a vertex, or neither) and $s^k\leq l^k$ (the number of removed vertices can not exceed the number of added vertices), and therefore $s^k\leq (k-1)/2$.

We arrive to case (c). In this case, \eqref{eq:OptimalStepSize} implies
$$\gamma^k_a=-\frac{\inprodmy{\nabla f(\bx^k)}{\bd^k}}{\rho \normmy{\bd^k}^2},$$
which combined with \eqref{eq:RatioFunctionStepSizes2} and \eqref{eq:RatioFunctionStepSizes3} results in
	\begin{equation}\label{eq:RatioFunctionNonMaximalStepSizes}
			f(\bx^{k+1})
		\leq f(\bx^k)+\gamma_a^k\inprodmy{\nabla f(\bx^k)}{\bd^k}+\frac{(\gamma_a^k)^2\rho}{2}\normmy{\bd^k}^2
		=f(\bx^k)-\frac{\inprodmy{\nabla f(\bx^k)}{\bd^k}^2}{2\rho \normmy{\bd^k}^2}.
	\end{equation}
From the algorithm's specific choice of $\bd^k$, we obtain that
\begin{equation}\label{eq:upperBoundOnInnerProd}
\begin{aligned}
0\geq \inprodmy{\nabla f(\bx^k)}{\bp^k-\bu^k}&=\inprodmy{\nabla f(\bx^k)}{\bp^k-\bx^k}+\inprodmy{\nabla f(\bx^k)}{\bx^k-\bu^k}\\
&\geq 2\inprodmy{\nabla f(\bx^k)}{\bd^k}.
\end{aligned}
\end{equation}
Applying the bound in \eqref{eq:upperBoundOnInnerProd} and the inequality $\norm{\bd^k}\leq D$ to \eqref{eq:RatioFunctionNonMaximalStepSizes}, it follows that
\begin{equation}\label{eq:RatioFunctionNonMaximalStepSizes2}
		f(\bx^{k+1})
		\leq f(\bx^k)-\frac{\inprodmy{\nabla f(\bx^k)}{\bd^k}^2}{2\rho \normmy{\bd^k}^2}\leq f(\bx^k)-\frac{\inprodmy{\nabla f(\bx^k)}{\bp^k-\bu^k}^2}{8\rho D^2}.
\end{equation}
By the definitions of $\bu^k$ and $\bp^k$, and since applying representation reduction procedure $\mathcal{R}$ ensures that that $|U^k|\leq N$, Corollary~\ref{cor:boundOnInnerProd} implies that for any $\bx^*\in X^*$,
\begin{equation}\label{eq:UpperBoundInnerProduct2}
\inprodmy{\nabla f(\bx^k)}{\bu^k-\bp^k}=\max_{\bp\in V, \bu\in U^k}\inprodmy{\nabla f(\bx^k)}{\bu-\bp}
\geq \frac{\Omega_X}{N}\frac{\inprodmy{\nabla f(\bx^k)}{\bx^k-\bx^*}}{\normmy{\bx^k-\bx^*}}.
\end{equation}
Lemma~\ref{lemma:ErrBound} implies that there exists $\bx^*\in X^*$ such that $\normmy{\bx^k-\bx^*}^2\leq \kappa (f(\bx^k)-f^*)$, which combined with
convexity of $f$, bounds  \eqref{eq:UpperBoundInnerProduct2} from below as follows:
\begin{equation*}\label{eq:UpperBoundInnerProduct3}
\begin{aligned}[b]
\inprodmy{\nabla f(\bx^k)}{\bu^k-\bp^k}^2
&\geq \left(\frac{\Omega_X}{N}\right)^2\frac{\inprod{\nabla f(\bx^k)}{\bx^k-\bx^*}^2}{\norm{\bx^k-\bx^*}^2}\\
&\geq \left(\frac{\Omega_X}{N}\right)^2\frac{(f(\bx^k)-f(\bx^*))^2}{\norm{\bx^k-\bx^*}^2}\\
&\geq \left(\frac{\Omega_X}{N}\right)^2\frac{(f(\bx^k)-f^*)^2}{\kappa(f(\bx^k)-f^*)}\\
&=\frac{(\Omega_X)^2}{N^2\kappa}(f(\bx^k)-f^*),
\end{aligned}
\end{equation*}
which along with \eqref{eq:RatioFunctionNonMaximalStepSizes2} yields
\begin{equation}\label{eq:RatioFunctionNonMaximalStepSizes3}
	\begin{aligned}[b]		
		f(\bx^{k+1})-f^*		
		&\leq f(\bx^k)-f^*-\frac{\inprodmy{\nabla f(\bx^k)}{\bu^k-\bp^k}^2}{8\rho D^2}\\
		&\leq (f(\bx^k)-f^*)\left(1-\frac{(\Omega_X)^2}{8\rho\kappa D^2N^2}\right)
\end{aligned}
\end{equation}
Therefore, if either of the cases (a) or (c) occurs, then by (\ref{666}) and \eqref{eq:RatioFunctionNonMaximalStepSizes3}, it follows that
\begin{equation} \label{718} f(\bx^{k+1})-f^* \leq (1-\alpha^{\dag}) (f(\bx^k)-f^*),\end{equation}
where $\alpha^{\dag}$ is defined in (\ref{defalpha}).  We can therefore conclude from cases (a)-(c) that until iteration $k$ we have at least $\frac{k-1}{2}$ iterations for which  (\ref{718}) holds, and therefore
\begin{equation} \label{eq:RatioFunctionNonMaximalStepSizes4}
	\begin{aligned}		
		f(\bx^{k})-f^*		
		&\leq (f(\bx^1)-f^*)(1-\alpha^\dag)^{(k-1)/2}.
\end{aligned}
\end{equation}
Applying Lemma~\ref{lemma:OFUpperBound} for $\bx=\bx^1$ we obtain $f(\bx^1)-f^*\leq C$, and the desired result (\ref{613}) follows.
\end{proof}

\subsection{Examples of Computing the Vertex-Facet Distance $\Omega_X$}\label{sec:FindingOmegaForPolyhedrons}
In this section, we demonstrate how to compute the vertex-facet distance constant $\Omega_X$ for a few simple polyhedral sets. We consider three sets: the unit simplex, the $\ell_1$ ball and the $\ell_\infty$ ball. We first describe each of the sets as a system of linear inequalities of the form $X=\paren{\bx:\bA\bx\leq \ba}$. Then, given the parameters $\bA$ and $\ba$, as well as the vertex set $V$, $\Omega_X$ can be computed by its definition, given by (\ref{defOmega}).

\emph{\underline{The unit simplex.}} The unit simplex $\Delta_n$ can be represented by
\begin{equation}
\begin{aligned}[b]
\bA=\begin{bmatrix}-\bI_{n \times n}\\\bone_n^T\\-\bone_n^T\end{bmatrix}\in \Real^{(n+2)\times n},\; \ba=\begin{bmatrix}\bzero_n\\1\\1\end{bmatrix}\in\Real^{(n+2)}.
\end{aligned}
\end{equation}
The set of extreme points is given by $V=\paren{\be_i}_{i=1}^n$. Notice that since there are only $n$ extreme points which are all affinely independent, using a rank reduction procedure which implements the Carath\'{e}odory theorem is the same as applying the trivial procedure that does not change the representation. In order to calculate $\Omega_X$, we first note that $I(V)=\{n+1,n+2\}$, and therefore
$$\varphi=\max_{i\in\paren{1,\ldots,n}} \norm{\bA_i}=\max_{i\in\paren{1,\ldots,n}} \norm{\be_i}=1$$
and
$$\zeta=\min\limits_{\bv\in \paren{\be_j}_{j=1}^n ,i\in\paren{1,\ldots,n}:-\inprod{\be_i}{\bv}<0}\inprod{\be_i}{\bv}=\min\limits_{i\in\paren{1,\ldots,n}}\norm{\be_i}^2=1,$$
which means that $\Omega_X=\frac{\zeta}{\varphi}=1$.

\emph{\underline{The $\ell_1$ ball.}} The $\ell_1$ ball is given by the set  $${X}=\paren{\bx\in\Real^n:\sum\limits_{i=1}^n |x_i|\leq 1}=\paren{\bx \in \Real^n:\inprod{\bw}{\bx}\leq 1,\forall \bw\in\paren{-1,1}^n}.$$
Therefore $\ba=\bone\in\Real^{2^n}$ and each row of the matrix $\bA\in \Real^{2^n\times n}$ is a vector in $\paren{-1,1}^n$. The set of extreme points is given by $V=\paren{\be_i}_{i=1}^n\bigcup \paren{-\be_i}_{i=1}^n$, and therefore has cardinality of $|V|=2n$.

Finally, we have that
$$\varphi=\max_{i\in\paren{1,\ldots,2^n}} \norm{\bA_i}=\sqrt{n}$$
and
\begin{equation*}
\begin{aligned}
\zeta&=\min_{\bv\in V,\bw\in\{-1,1\}^n: \inprod{\bv}{\bw}<1}(1-\inprod{\bv}{\bw})\\
&=\min_{i\in\paren{1,\ldots,n},\;\bw\in\{-1,1\}^n: \inprod{\be_i}{\bw}<1}(1-\inprod{\be_i}{\bw})\\
&=\min_{\bw\in\{-1,1\}^n}(1+|w_i|)=2,
\end{aligned}
\end{equation*}
which means that $\Omega_X=\frac{\zeta}{\varphi}=\frac{2}{\sqrt{n}}$.

\emph{\underline{The $\ell_\infty$ ball.}} The $\ell_\infty$ ball is represented by
\begin{equation}
\begin{aligned}[b]
\bA=\begin{bmatrix}\bI \\-\bI \end{bmatrix}\in \Real^{2n\times n},\; \ba=\begin{bmatrix}\bone\\\bone\end{bmatrix}\in\Real^{2n}.
\end{aligned}
\end{equation}
The set of extreme points is given by $V=\paren{-1,1}^n$, which in particular implies that $|V|=2^n$. Therefore, for large-scale problems, using the representation reduction procedure, which is based on Carath\'{e}odory theorem, is crucial in order to obtain a practical implementation.

From the definition of $\bA$ and $V$, it follows that
$$\varphi=\max_{i\in\paren{1,\ldots,2n}} \norm{\bA_i}=\max_{i\in\paren{1,\ldots,n}} \norm{\be_i}=1$$
and
$$\zeta=\min_{i\in\paren{1,\ldots,n},\;\bv\in\{-1,1\}^n: \inprod{\be_i}{\bv}<1}{(1-\inprod{\be_i}{\bv})}=2,$$
which implies that $\Omega_X=\frac{\zeta}{\varphi}=2$.
\bibliographystyle{abbrv}

\begin{appendices}
\section{Incremental Representation Reduction using the Carath\'{e}odory Theorem}\label{appx:Caratheodory}
In this section we will show a way to efficiently and incrementally implement the constructive proof of Carath\'{e}odory theorem, as part of the VRU scheme, at each iteration of the ASCG algorithm. We note that this reduction procedure does not have to be employed, and instead the trivial procedure, which does not change the representation can be used. In that case, the upper bound on the number of extreme points in the representation is just the number of extreme points of the feasible set $X$.  \\
\indent The implementation described in this section will allow maintaining a vertex representation set $U^k$, with cardinality of at most $n+1$, at a computational cost of $O(n^2)$ operations per iteration. For this purpose, we assume that at the beginning of iteration $k$, $\bx^{k}$ has a representation with vertex set $U^{k}=\paren{\bv^1,\ldots,\bv^{L}}\subseteq V$, such that the vectors in the set are affinely independent. Moreover, we assume that at the beginning of iteration $k$, we have at our disposal two matrices $\bT^k\in\Real^{n\times n}$ and ${\bW}^k\in\Real^{n\times(L-1)}$. We define $\bV^k\in \Real^{n\times(L-1)}$ to be the matrix whose $i$th column is the vector $\bw^i=\bv^{i+1}-\bv^1$ for $i=1, \ldots, L-1$, where $\bv^1$ is called the reference vertex. The matrix $\bT^k$ is a product of elementary matrices, which ensures that the matrix ${\bW}^k=\bT^k\bV^k$ is in row echelon form. The implementation does not require to save the matrix $\bV^k$, and so at each iteration, only the matrices $\bT^k$ and $\bW^k$ are updated.

 Let $U^{k+1}$ be the vertex set and let $\bmu^{k+1}$ be the coefficients vector at the end of iteration $k$, before applying the rank reduction procedure.  Updating the matrices $\bW^{k+1}$ and $\bT^{k+1}$, as well as $U^{k+1}$ and $\bmu^{k+1}$, is done according to the following \emph{Incremental Representation Reduction} scheme, which is partially based on the proof of Carath\'{e}odory theorem presented in \cite[Section 17]{Rockafellar1970}.
\begin{framed}\label{alg:Caratheodory}
	\setstretch{1.2}
	\noindent\underline{\bf Incremental Representation Reduction (IRR)}\\
	{\bf Input}: Representation $(U^{k+1},\bmu^{k+1})$ of point $\bx^{k+1}$, set $U^k=\paren{\bv^1,\ldots,\bv^{L}}$ of affinely\\ \indent~~~~~independent	vectors, and matrices $\bT^{k}\in\Real^{n\times n}$ and $\bW^{k}\in\Real^{n\times (L-1)}$.\\
	{\bf Output}: Updated representation $(U^{k+1},\bmu^{k+1})$ of $\bx^{k+1}$, and matrices $\bT^{k+1}\in\Real^{n\times n}$\\
	\indent~~~~~ and $\bW^{k+1}\in\Real^{n\times (|U_{k+1}|-1)}$.
	\vspace{\topsep}
	\begin{enumerate}
	\itemsep0em
	\item Set $L:=|U^{k}|$.
	\item Update $\bT^{k+1}:=\bT^k$.
	\item If $|U^{k+1}|=1$, then set the matrix $\bW^{k+1}$ to be empty and $\bT^{k+1}:=\bI$.
	\item Else, if $|U^{k+1}|= L$, then set $\bW^{k+1}:=\bW^{k}$.
	\item\label{step:IRR_DropStep} Else, if $|U^{k+1}|= L-1>1$ (drop step), then
		\vspace{-\topsep}
		\begin{enumerate}[label=(\alph*), ref=\arabic{enumi}(\alph*)]
			\itemsep0em
			\item Find $i^*\in\paren{1,\ldots,L}$ such that $\bv^{i^*}\in U^k/U^{k+1}$.
			\item \label{step:IRR_ChangeReferenceVertex} If $i^*=1$ (the reference vertex was removed), then
			remove the first column of $\bW^k$ and change reference vertex to $\bv^2$, using the update formula
			 $${\bW^{k+1}:=\bW^k\begin{bmatrix}\bzero& \bI_{(L-1)\times (L-1)}\end{bmatrix}^T+\bT^k(\bv^1-\bv^2)\bone^T},$$
			  where  $\bone,\bzero\in \Real^{L-1}$.	
			\item Else (a non-reference vertex was removed), remove column $i^*-1$ from $\bW^{k+1}$.
		\end{enumerate}
	\item Else, if $|U^{k+1}|=L+1$ (forward step), then
	\vspace{-\topsep}
	\begin{enumerate}[label=(\alph*), ref=\arabic{enumi}(\alph*)]
		\itemsep0em
		\item Find $\bv^{L+1}\in U^{k+1}/U^k$.
		\item Compute $\bw^{L}:=\bv^{L+1}-\bv^1$.
		\item Update the matrix $\bW^{k+1}:=[\bW^k,\bT^k\bw^{L}]$.
		\item \label{step:IRR_FindRowRank}Compute $M$ - the row rank of ${\bW}^{k+1}$.
		\item If $L>M$, then
		\vspace{-\topsep}
		\begin{enumerate}[label=\roman*., ref=\arabic{enumi}(\alph{enumii})\roman*]
		\itemsep0em
		\item \label{step:IRR_SolSysEq} Find a solution $\blambda$ of the following system \\
		$${\bW}^{k+1}\blambda=\bzero,\; \lambda_{L}=-1.$$
		\item Set the vector $\tilde{\blambda}\in \Real^{L+1}$ to be \\		
		$$\tilde{\lambda}_i:=\begin{cases}
		-\sum_{i=2}^{L+1} \lambda_{i-1} & i=1\\
		\lambda_{i-1}& i=2,\ldots,L+1\\
		\end{cases}.$$
		\item Compute $\overline{\alpha}:=\min_{i:\tilde{\lambda}_i<0} -\frac{\mu^k_i}{\tilde{\lambda}_i}$ and $\underline{\alpha}:=\max_{i:\tilde{\lambda}_i>0} -\frac{\mu^k_i}{\tilde{\lambda}_i}$ and set\\		
		$$\alpha=\begin{cases}
		\overline{\alpha} & \tilde{\lambda}_1\geq 0\\
		\underline{\alpha} & \tilde{\lambda}_1<0.
		\end{cases}.$$		
		\item Update $\mu^{k+1}_{\bv^i}:=\mu^{k+1}_{\bv^i}+\alpha \tilde{\lambda}_i$ for all $i=1,\ldots, L+1$.
		\item \label{step:IRR_FindNullSet} Compute $I=\paren{i\in\paren{1,\ldots,L+1}:\mu^{k+1}_{\bv^i}=0}$.
		\item \label{step:IRR_RemoveVertices} For each $i\in I$ remove column $i-1$ matrix $\bW^{k+1}$.
		\item \label{step:IRR_UpdateRepresentation} Update $U^{k+1}=U^{k+1}/\paren{\bv_i}_{i\in I}$.
		\end{enumerate}
	\end{enumerate}
	\item \label{step:IRR_RowElimination}If ${\bW}^{k+1}$ is not in row echelon form, then construct a matrix $\tilde{\bT}$, as a composition of elementary matrices, such that $\tilde{\bT}{\bW}^{k+1}$ is row echelon form, and update ${\bW}^{k+1}:=\tilde{\bT}{\bW}^{k+1}$ and $\bT^{k+1}:=\tilde{\bT}\bT^{k+1}$.				
	\end{enumerate}	
	\vspace{-\topsep}
\end{framed}
Notice that in order to compute the row rank of the matrix $\bW^{k+1}$ in step~\ref{step:IRR_FindRowRank}, we may simply convert the matrix to row echelon form, and then count the number of nonzero rows. This is done similarly to step~\ref{step:IRR_RowElimination}, and requires ranking of at most one column. We will need to rerank the matrix in step~\ref{step:IRR_RowElimination} only if $L>M$, and subsequently at least one column is removed in step~\ref{step:IRR_RemoveVertices}.

The IRR scheme may reduce the size of the input $U^{k+1}$ only in the case of a forward step, since otherwise the vertices in $U^{k+1}$ are all affinely independent. Nonetheless, the IRR scheme \emph{must} be applied at each iteration in order to maintain the matrices ${\bW}^k$ and $\bT^k$.

The efficiency of the scheme relies on the fact that only a small number of vertices are either added to or removed from the representation. The potentially computationally expensive steps are: step \ref{step:IRR_ChangeReferenceVertex} - replacing the reference vertex, step~\ref{step:IRR_FindRowRank} - finding the row rank of $\bW^{k+1}$, step \ref{step:IRR_SolSysEq} - solving the system of linear equalities, step \ref{step:IRR_RemoveVertices} - removing columns corresponding with the vertices eliminated from the representation, and step \ref{step:IRR_RowElimination} - the ranking of the resulting matrix ${\bW}^{k+1}$. Step \ref{step:IRR_ChangeReferenceVertex} can be implemented without explicitly using matrix multiplication and therefore has a computational cost of $O(n^2)$. Since ${\bW}^k$ was in row echelon form, step~\ref{step:IRR_FindRowRank} requires a row elimination procedure, similar to step~\ref{step:IRR_RowElimination}, to be conducted only on the last column of  ${\bW}^{k+1}$, which involves at most $O(n)$ operations and an additional $O(n^2)$ operation for updating $\bT^{k+1}$. Moreover, since ${\bW}^k$ was full column rank, the IRR scheme guarantees that in step $\ref{step:IRR_SolSysEq}$ the vector $\blambda$ has a unique solution, and since ${\bW}^{k+1}$ is in row echelon form, it can be found in $O(n^2)$ operations. Moreover, in step \ref{step:IRR_RemoveVertices}, the specific choice of $\alpha$ ensures that the reference vertex $\bv^1$ is not eliminated from the representation, and so there is no need to change the reference vertex at this stage. Furthermore, it is reasonable to assume that the set $I$ satisfies $|I|=O(1)$, since otherwise the vector $\bx^{k+1}$, produced by a forward step, can be represented by significantly less vertices than $\bx^k$, which, although possible, is numerically unlikely.  Therefore, assuming that indeed $|I|=O(1)$, the matrix $\tilde{\bT}$, calculated in step \ref{step:IRR_RowElimination}, applies a row elimination procedure to at most $O(1)$ rows (one for each column removed from $\bW^{k+1}$) or one column (if a column was added to $\bW^{k+1}$). Conducting such an elimination on either row or column takes at most $O(n^2)$ operations, which may include row switching and at most $n$ row addition and multiplication. Therefore, the total computational cost of the IRR scheme amounts to $O(n^2)$.
\end{appendices}

\end{document}